\documentclass{amsart}

\usepackage{latexsym,enumerate}
\usepackage{amsmath,amsthm,amsopn,amstext,amscd,amsfonts,amssymb}
\usepackage[ansinew]{inputenc}
\usepackage{verbatim}
\usepackage{wasysym}
\usepackage[all]{xy}
\usepackage{epsfig} 
\usepackage{graphicx,psfrag} 
\usepackage{subfigure}
\usepackage{pstricks,pst-node}

\usepackage{multicol}
\usepackage{color}
\usepackage{colortbl}

\newcommand{\RR}{\mathbb{R}}
\newcommand{\QQ}{\mathbb{Q}}
\newcommand{\ZZ}{\mathbb{Z}}

\newtheorem{theorem}{Theorem}[section]
\newtheorem{lemma}[theorem]{Lemma}
\newtheorem{proposition}[theorem]{Proposition}
\newtheorem{corollary}[theorem]{Corollary}

\newcommand{\p}{\partial}

\renewcommand{\a}{\alpha}

\renewcommand{\d}{\delta}
\newcommand{\D}{\Delta}

\newcommand{\G}{\Gamma}

\renewcommand{\L}{\Lambda}

\begin{document}
\DeclareGraphicsExtensions{.jpg,.pdf,.mps,.png}
\title{Geometric-Arithmetic index and line graph}

\author[Domingo Pestana]{Domingo Pestana$^{(1)}$}
\address{Departamento de Matem\'aticas, Universidad Carlos III de Madrid,
Avenida de la Universidad 30, 28911 Legan\'es, Madrid, Spain}
\email{dompes@math.uc3m.es}
\thanks{$^{(1)}$ Supported in part by two grants from Ministerio de Econom{\'\i}a y Competitividad, Agencia Estatal de
Investigación (AEI) and Fondo Europeo de Desarrollo Regional (FEDER) (MTM2016-78227-C2-1-P and MTM2015-69323-REDT), Spain.}

\author[Jos\'e M. Sigarreta]{Jos\'e M. Sigarreta$^{(1)(2)}$}
\address{Facultad de Matem\'aticas, Universidad Aut\'onoma de Guerrero,
Carlos E. Adame No.54 Col. Garita, 39650 Acalpulco Gro., Mexico}
\email{jsmathguerrero@gmail.com}
\thanks{$^{(2)}$ Supported in part by a grant from CONACYT (FOMIX-CONACyT-UAGro 249818), M\'exico.}

\author[Eva Tour{\'\i}s]{Eva Tour{\'\i}s$^{(1)}$}
\address{Departamento de Matem\'aticas, Facultad de Ciencias, Universidad Aut\'onoma de Madrid, Campus de Cantoblanco, 28049 Madrid, Spain}
\email{eva.touris@uam.es}

\date{\today}

\maketitle{}


\begin{abstract}
The concept of geometric-arithmetic index was introduced in the
chemical graph theory recently, but it has shown to be useful.
The aim of this paper is to obtain new inequalities involving the geometric-arithmetic index $GA_1$
and characterize graphs extremal with respect to them.
Besides, we prove inequalities involving the geometric-arithmetic index of line graphs.
\end{abstract}



\section{Introduction}

A single number, representing a chemical structure in graph-theoretical terms via the
molecular graph, is called a topological descriptor and if it in addition correlates with
a molecular property it is called topological index, which is used to understand physicochemical
properties of chemical compounds.
Topological indices are interesting since they capture some of the properties of a molecule in a single number.
Hundreds of topological indices have been introduced and studied, starting with the
seminal work by Wiener in which he used the sum of all shortest-path distances of
a (molecular) graph for modeling physical properties of alkanes. 

Topological indices based on end-vertex degrees of edges have been
used over 40 years. Among them, several indices are recognized to be useful tools in
chemical researches. Probably, the best know such descriptor is the Randi\'c connectivity
index ($R$). 

Two of the main successors of the Randi\'c index are the first and second Zagreb indices,
denoted by $M_1$ and $M_2$, respectively, and introduced by Gutman and Trinajsti\'c in $1972$ (see \cite{GT}).
They are defined as
$$
M_1(G) = \sum_{u\in V(G)} d_u^2,
\qquad
M_2(G) = \sum_{uv\in E(G)} d_u d_v ,
\qquad
$$
where $uv$ denotes the edge of the graph $G$ connecting the vertices $u$ and $v$, and
$d_u$ is the degree of the vertex $u$.

There is a vast amount of research on the Zagreb indices.
For details of their chemical applications and mathematical theory see \cite{Gutman}, \cite{GD}, 
and the references therein.

In \cite{LZheng}, \cite{LZhao}, \cite{MN}, the \emph{first and second variable Zagreb indices} are defined as
$$
M_1^{\a}(G) = \sum_{u\in V(G)} d_u^{\a},
\qquad
M_2^{\a}(G) = \sum_{uv\in E(G)} (d_u d_v)^\a ,
$$
with $\a \in \RR$.

The concept of variable molecular descriptors was proposed as a new way of
characterizing heteroatoms in molecules (see \cite{R2}), 
but also to assess the structural differences (e.g.,
the relative role of carbon atoms of acyclic and cyclic parts in alkylcycloalkanes \cite{RPL}).
The idea behind the variable molecular descriptors is that the variables are determined during the
regression so that the standard error of estimate for a particular studied property is as small as possible (see, e.g., \cite{MN}).

In the paper of Gutman and To\v{s}ovi\'c \cite{Gutman8}, the correlation abilities of $20$ vertex-degree-based topological indices
occurring in the chemical literature were tested for the case of standard
heats of formation and normal boiling points of octane isomers.
It is remarkable to realize that the second variable Zagreb index $M_2^\alpha$
with exponent $\alpha = -1$ (and to a lesser extent with exponent $\alpha = -2$)
performs significantly better than the Randi\'c index ($R=M_2^{-1/2}$).

The second variable Zagreb index is used in the structure-boiling point modeling of benzenoid hydrocarbons \cite{NMTJ}.
Also, variable Zagreb indices exhibit a potential applicability for deriving multi-linear regression models. 
Various properties and relations of these indices are discussed in several papers (see, e.g., \cite{AP}, \cite{LZhao}, \cite{LL}, \cite{SDGM}). 

Note that $M_1^{2}$ is the first Zagreb index $M_1$, $M_1^{-1}$ is the inverse index $ID$, 
$M_1^{3}$ is the forgotten index $F$, etc.; 
also, $M_2^{-1/2}$ is the usual Randi\'c index, $M_2^{1}$ is the second Zagreb index $M_2$,
$M_2^{-1}$ is the modified Zagreb index, etc. 


The \emph{general sum-connectivity index} was defined by Zhou and Trinajsti\'c in \cite{ZT2} as
$$
\chi_{\a}(G) = \sum_{uv\in E(G)} (d_u+ d_v)^\a .
$$
Note that $\chi_{_{1}}$ is the first Zagreb index $M_1$, $2\chi_{_{-1}}$ is the harmonic index $H$,
$\chi_{_{-1/2}}$ is the sum-connectivity index $\chi$, etc.

The first geometric-arithmetic index $GA_1$ is defined in \cite{VF} as
$$
GA_1 = GA_1(G) = \sum_{uv\in E(G)}\frac{\sqrt{d_u d_v}}{\frac12 (d_u + d_v)} \,.
$$
Although $GA_1$ was introduced in $2009$, there are many papers dealing with this index
(see, e.g., \cite{D}, \cite{DGF}, \cite{DGF2}, \cite{MH}, \cite{RS3}, 
\cite{VF} and the references therein).
There are other geometric-arithmetic indices, like $Z_{p,q}$ ($Z_{0,1} = GA_1$), but the results in \cite[p.598]{DGF}
show that the $GA_1$ index gathers the
same information on observed molecules as other $Z_{p,q}$ indices.

Although only about 1000 benzenoid hydrocarbons are known, the number of
possible benzenoid hydrocarbons is huge. For instance, the number of
possible benzenoid hydrocarbons with 35 benzene rings is $5.85\cdot 10^{21}$ \cite{VGJ}.
Therefore, modeling their physico-chemical properties is important in order
to predict properties of currently unknown species.
The predicting ability of the $GA_1$ index compared with
Randi\'c index is reasonably better (see \cite[Table 1]{DGF}).
The graphic in \cite[Fig.7]{DGF} (from \cite[Table 2]{DGF}, \cite{TRC}) shows
that there exists a good linear correlation between $GA_1$ and the heat of formation of benzenoid hydrocarbons
(the correlation coefficient is equal to $0.972$).

Furthermore, the improvement in
prediction with $GA_1$ index comparing to Randi\'c index in the case of standard
enthalpy of vaporization is more than 9$\%$. That is why one can think that $GA_1$ index
should be considered in the QSPR/QSAR researches.

Line graphs were initially introduced in the papers \cite{W} and \cite{Kr},
although the terminology of line graph was used in \cite{HN} for the first time.
They are an active topic of research at this moment.
Some topological indices of line graphs have been considered previously in \cite{NKMT}, \cite{RLC} and \cite{SX}.
The \emph{line graph} $\mathcal{L}(G)$ of $G$ is a graph whose vertices are the edges of $G$, and two vertices and are two vertices are incident if and only if they have a common end vertex in $G$.

A main topic in the study of topological indices is to find bounds of the indices involving several parameters.
The aim of this paper is to obtain new inequalities involving the geometric-arithmetic index $GA_1$
and characterize graphs extremal with respect to them.
Besides, we prove inequalities involving the geometric-arithmetic index of line graphs.

Throughout this work, $G=(V (G),E (G))$ denotes a (non-oriented) finite simple (without multiple edges and loops) graph such that each connected connected component of $G$ has at least an edge.
Given a graph $G$ and $v\in V(G)$, we denote by $N(v)$ the set of \emph{neighbors} of $v$, i.e.,
$N(v)=\{u\in V(G)|\, uv\in E(G) \}$.
We denote by $\D,\d,n,m$ the maximum degree, the minimum degree and the cardinality of the set of vertices and edges of $G$, respectively.
Also, we denote by $\D_{\mathcal{L}(G)},\d_{\mathcal{L}(G)},n_{\mathcal{L}(G)},m_{\mathcal{L}(G)}$ the maximum degree, the minimum degree and the cardinality of the set of vertices and edges of the line graph $\mathcal{L}(G)$ of $G$, respectively.
By $G_1 \approx G_2$, we mean that the graphs $G_1$ and $G_2$ are isomorphic.
We say that a graph $G$ is \emph{non-trivial} if each connected connected component of $G$ has at least two edges.
Since $\mathcal{L}(P_2)$ is a single vertex without edges,
in order to work with line graphs we just consider non-trivial graphs.

\section{Small values of the Geometric-Arithmetic index}

To obtain the firsts minimum and maximum values of the geometric-arithmetic index
of some classes of graphs
is an interesting topic (see, e.g., \cite{DEB}, \cite{DZT}, \cite{HHD}).
In this section, we find the graphs having the firsts minimum values of this index in Theorem \ref{t:line3}.
We need first some technical results.

\medskip

Recall that a $(\D,\d)$-\emph{biregular graph} (or simply a biregular graph) is a bipartite graph for which any vertex in one side of the given bipartition has degree $\D$
and any vertex in the other side of the bipartition has degree $\d$.

In \cite{D} (see also \cite[p.609-610]{DGF}) we find the following bounds.

\begin{proposition} \label{p:0}
If $G$ is a graph with $m$ edges, maximum degree $\D$ and minimum degree $\d$, then
\begin{equation} 
\frac{2m\sqrt{\D\d}}{\D + \d} \le GA_1(G) \le m .
\end{equation}
The equality in the first inequality is attained if and only if $G$ is regular or biregular.
The equality in the second inequality is attained if and only if $G$ is regular.
\end{proposition}

In \cite[Lemma 2.2 and Corollary 2.3]{RS2} we find the following results.

\begin{lemma} \label{l:t}
Let $f$ be the function $f(t)=\frac{2t}{1+t^2}$ on the interval $[0,\infty)$. Then $f$
strictly increases in $[0,1]$, strictly decreases in $[1,\infty)$, $f(t)=1$ if and only if $t=1$ and $f(t)=f(t_0)$ if and only if either $t=t_0$ or $t=t_0^{-1}$.
\end{lemma}

\begin{corollary} \label{c:t}
Let $g$ be the function $g(x,y)=\frac{2\sqrt{xy}}{x + y}$ with $0<a\le x,y \le b$. Then
$$
\frac{2\sqrt{ab}}{a + b} \le g(x,y) \le 1.
$$
The equality in the lower bound is attained if and only if either $x=a$ and $y=b$, or $x=b$ and $y=a$,
and the equality in the upper bound is attained if and only if $x=y$.
Besides, $g(x,y)= g(x',y')$ if and only if $x/y$ is equal to either $x'/y'$ or $y'/x'$.
\end{corollary}

\begin{proposition} \label{p:nostar}
If $G$ is a graph with $n$ vertices, $m$ edges and maximum degree $\D$ such that $\D \le n-2$,
then
$$
GA_1(G) > \frac{2m\sqrt{n-2}}{n-1}  \ge 2 \, \sqrt{n-2}\, .
$$
\end{proposition}

\begin{proof}
Since $1 \le \d \le \D \le n-2$, Proposition \ref{p:0} and Corollary \ref{c:t} give
$$
GA_1(G)
\ge \frac{2m\sqrt{\D\d}}{\D + \d}
\ge \frac{2m\sqrt{n-2}}{n-1}\, ,
$$
and $GA_1(G) = \frac{2m\sqrt{n-2}}{n-1}$ if and only if $n=3$ and $G$ is $1$-regular or $n \ge 4$ and $G$ is $(n-2,1)$-biregular.
It is clear that there does not exist any $1$-regular graph with $n=3$.
Thus, we just need to prove that there does not exist any $(n-2,1)$-biregular graph.
Seeking for a contradiction assume that $G$ is a $(n-2,1)$-biregular graph.
Let $V(G)=\{v_1,v_2,\dots,v_{n-1},v_n\}$ with $d_{v_1}=n-2$ and $N(v_1)=\{v_2,\dots,v_{n-1}\}$.
Since $G$ is $(n-2,1)$-biregular, $d_{v_j}=1$ for every $2 \le j \le n-1$ and so, $N(v_n)=\emptyset$, a contradiction.
This finishes the proof.
%
\end{proof}

We say that an edge in a graph is \emph{pendant} if one of its endpoints has degree $1$.

In the study of any parameter of graphs it is interesting to determine the graphs for which this parameter has small values.
The following theorem characterizes the graphs with small geometric-arithmetic index.

\begin{theorem} \label{t:line3}
Let $G$ be a connected graph.

$(1)$ If $GA_1(G) \le 1$, then $G$ is isomorphic to $P_2$ and $GA_1(G) = 1$.

$(2)$ If $1 < GA_1(G) \le 2$, then $G$ is isomorphic to $P_3$ and $GA_1(G) = \frac{4\sqrt{2}}3\,$.

$(3)$ If $2 < GA_1(G) \le 3$, then $G$ is isomorphic to either $C_3$, $P_4$ or $S_4$, and $GA_1(G)$ takes the values
$3,\, 1+\frac{4\sqrt{2}}3\,,\, \frac{3\sqrt{3}}2\,$, respectively.

$(4)$ If $3 < GA_1(G) \le 4$, then $G$ is isomorphic to either $S_6$,
$S_5$, $P_5$, $S_{1,2}$,
$C_4$, or $C_3$ with a pendant edge,
and $GA_1(G)$ takes the values
$\frac{5\sqrt{5}}3\,,\,\frac{16}5\,, \, 2 +\frac{4\sqrt{2}}3\,,\, \sqrt{3} + \frac{2\sqrt{2}}{3} + \frac{2\sqrt{6}}{5}\,,
\,4,\, 1+\frac{4\sqrt{6}}5 + \frac{\sqrt{3}}2
\,,$ respectively.
\end{theorem}

\begin{proof}
If $GA_1(G) \le k$ for some fixed $k$, then \eqref{eq23} and Theorem \ref{t:end} give
\begin{equation} \label{eq01}
\frac{2(n-1)^{3/2}}{n}\le k ,
\qquad
m \le \frac{kn}{2\sqrt{n-1}}\, .
\end{equation}

If $GA_1(G) \le 1$, then the first inequality in \eqref{eq01} implies $n\le 2$.
Hence, $G$ is isomorphic to $P_2$ and $GA_1(G) = 1$.

If $1 < GA_1(G) \le 2$, then the first inequality in \eqref{eq01} implies $n\le 3$.
Since $1 < GA_1(G)$, we have $n= 3$.
Hence, $G$ is isomorphic to either $P_3$ or $C_3$.
Since $GA_1(C_3) = 3$, $G$ is isomorphic to $P_3$ and $GA_1(G) = \frac{4\sqrt{2}}3\,$.

If $2 < GA_1(G) \le 3$, then the first inequality in \eqref{eq01} implies $n\le 4$.
Since $2 < GA_1(G)$, we have that either $G$ is isomorphic to $C_3$ (and $GA_1(C_3) = 3$) or $n= 4$.
If $n=4$, then the second inequality in \eqref{eq01} implies $m\le 3$ and $G$ is a tree.
Hence, $G$ is isomorphic to either $P_4$ or $S_4$, and we have $GA_1(P_4) = 1+\frac{4\sqrt{2}}3$
or $GA_1(S_4)= \frac{3\sqrt{3}}2\,$.

If $3 < GA_1(G) \le 4$, then the first inequality in \eqref{eq01} implies $n\le 6$.

If $n=4$, then the second inequality in \eqref{eq01} implies $m\le 4$.
Since $3 < GA_1(G)$, $G$ is not a tree and $m=4$. Thus,
$G$ is isomorphic to either $C_4$ or $C_3$ with a pendant edge, and
$GA_1(G)$ is equal to
$4,\, 1+\frac{4\sqrt{6}}5 + \frac{\sqrt{3}}2\,$, respectively.

If $n=5$, then the second inequality in \eqref{eq01} implies $m\le 5$.
Seeking for a contradiction assume that $m = 5$.
By Theorem \ref{t:end},
$$
4 \ge GA_1(G) \ge \frac{2m\sqrt{n-1}}{n} = 4,
$$
and then $GA_1(G) = \frac{2m\sqrt{n-1}}{n}$. Since the equality in Theorem \ref{t:end} is attained, $G$ is a star graph
and this contradicts $n=m=5$.
Therefore, $m\le 4$ and $G$ is a tree.
By \cite[Theorem 3]{VF}, we know that $GA_1(S_n) \le GA_1(G) \le GA_1(P_n)$ for every tree $G$ with $n$ vertices.
Hence,
$$
3 < \frac{16}5 = GA_1(S_5) \le GA_1(G) \le GA_1(P_5) = 2 +\frac{4\sqrt{2}}3 < 4
$$
for every tree $G$ with $5$ vertices. Thus,
$G$ is isomorphic to either $S_5$, $P_5$ or the double star graph $S_{1,2}$,
and
$GA_1(G)$ is equal to
$\frac{16}5\,, \, 2 +\frac{4\sqrt{2}}3\,,\, \sqrt{3} + \frac{2\sqrt{2}}{3} + \frac{2\sqrt{6}}{5}\,,$ respectively.

If $n=6$, then the second inequality in \eqref{eq01} implies $m\le 5$ and $G$ is a tree.
We have $GA_1(S_6)= \frac{5\sqrt{5}}3\,$.
If $G$ is not isomorphic to $S_6$,
then $\D \le n-2$ and Proposition \ref{p:nostar} gives
$$
GA_1(G) > 2 \, \sqrt{n-2} = 4 .
$$
Hence, $G$ is isomorphic to $S_6$ and $GA_1(G)= \frac{5\sqrt{5}}3\,$.
\end{proof}

By using Theorem \ref{t:line3}, it is clear that one can obtain a similar result for non-connected graphs.

Also, the following result is a version of Theorem \ref{t:line3} for line graphs.

\begin{proposition} \label{t:line4}
Let $G$ be a non-trivial connected graph.

$(1)$ If $GA_1(\mathcal{L}(G)) \le 1$, then $G$ is isomorphic to $P_3$ and $GA_1(\mathcal{L}(G)) = 1$.

$(2)$ If $1 < GA_1(\mathcal{L}(G)) \le 2$, then $G$ is isomorphic to $P_4$ and $GA_1(\mathcal{L}(G)) = \frac{4\sqrt{2}}3\,$.

$(3)$ If $2 < GA_1(\mathcal{L}(G)) \le 3$, then $G$ is isomorphic to either $C_3$, $S_4$ or $P_5$, and $GA_1(\mathcal{L}(G))$ takes the values
$3,\, 1+\frac{4\sqrt{2}}3\,$.

$(4)$ If $3 < GA_1(\mathcal{L}(G)) \le 4$, then $G$ is isomorphic to either
$P_6$,
$C_4$ or $S_{1,2}$,
and $GA_1(\mathcal{L}(G))$ takes the values
$2 +\frac{4\sqrt{2}}3\,,
\,4,\, 1+\frac{4\sqrt{6}}5 + \frac{\sqrt{3}}2
\,,$ respectively.
\end{proposition}

\begin{proof}
It suffices to apply Theorem \ref{t:line3}, finding the graphs such their line graphs appear in the statement of Theorem \ref{t:line3}.
Note that there does not exist any graph $G$ with $\mathcal{L}(G)=S_{1,2}$ or $\mathcal{L}(G)=S_n$ for $n\ge 4$.
Also, $\mathcal{L}(C_3)= \mathcal{L}(S_4)=C_3$.
\end{proof}

A natural problem in the study of any topological index is its monotonicity with respect to deletion of edges.

\smallskip

We say that $u_0v_0\in E(G)$ is \emph{minimal} if
$d_{u_0} \le d_{v}$ for every $v \in N(u_0) \setminus \{v_0 \}$ and
$d_{v_0} \le d_{v}$ for every $v \in N(v_0) \setminus \{u_0 \}$.
(Note that every edge $u_0v_0\in E(G)$ with $d_{u_0} = d_{v_0} =\d$ is minimal.)

\begin{theorem} \label{t:edge}
If $u_0v_0$ is a minimal edge in the graph $G$, then
$$
GA_1(G\setminus \{u_0v_0 \}) < GA_1(G) - \frac{2\sqrt{d_{u_0} d_{v_0}}}{d_{u_0} + d_{v_0}} \,.
$$
\end{theorem}

\begin{proof}
Denote by $d_u'$ the degree of $u \in V(G)=V(G\setminus \{u_0v_0 \})$ in the graph $G\setminus \{u_0v_0 \}$.

If $uv \in E(G)$ and if $\{u,v\}\cap \{u_0,v_0\}=\emptyset$, then $d_u'=d_u$, $d_v'=d_v$ and
$$
\frac{2\sqrt{d_{u_0}' d_v'}}{d_{u_0}' + d_v'}
= \frac{2\sqrt{d_{u_0} d_v}}{d_{u_0} + d_v}
\, .
$$

If $v \in N(u_0) \setminus \{v_0 \}$, then
$d_{u_0}'= d_{u_0}-1< d_{u_0} \le d_{v}$ and Corollary \ref{c:t} gives
$$
\frac{2\sqrt{d_{u_0}' d_v'}}{d_{u_0}' + d_v'}
= \frac{2\sqrt{(d_{u_0}-1) d_v}}{d_{u_0} + d_v-1}
< \frac{2\sqrt{d_{u_0} d_v}}{d_{u_0} + d_v}
\, .
$$

In a similar way, if $v \in N(v_0) \setminus \{u_0 \}$, then
$$
\frac{2\sqrt{d_{v_0}' d_v'}}{d_{v_0}' + d_v'}
= \frac{2\sqrt{(d_{v_0}-1) d_v}}{d_{v_0} + d_v-1}
< \frac{2\sqrt{d_{v_0} d_v}}{d_{v_0} + d_v}
\, .
$$

Summing up on $uv \in E(G) \setminus \{u_0v_0 \}$ we obtain the inequality.
\end{proof}

Another natural problem is to characterize the graphs $G$ such that $GA_1(G)$ is a rational number.

\begin{proposition} \label{p:square}
A graph $G$ satisfies
$GA_1(G) \in \QQ$ if and only if $d_u d_v$ is a perfect square for every $uv\in E(G)$.
\end{proposition}

Proposition \ref{p:square} is a direct consequence of the following known result (since $1$ is a squarefree positive integer).

\begin{proposition} \label{p:square2}
The set $\{  \sqrt{n}\,: \; n$ is a squarefree positive integer$\,\}$
is linearly independent over $\QQ$.
\end{proposition}

One can think that perhaps condition in Proposition \ref{p:square} holds if and only if $G$ is a biregular graph.
However, the following example shows that this is not true:

\medskip

\noindent {\bf Example 5}.
Given any different positive integers $n_1,n_2,$
the double star graph $S_{n_1^2-1,n_2^2-1}$
verifies
$$
GA_1(S_{n_1^2-1,n_2^2-1}) = \frac{2(n_1^2-1)n_1}{n_1^2+1} + \frac{2(n_2^2-1)n_2}{n_2^2+1} + \frac{2\,n_1 n_2}{n_1^2+n_2^2} \in \QQ ,
$$
and $S_{n_1^2-1,n_2^2-1}$ is not a biregular graph since $n_1\neq n_2$.

\medskip

Theorem \ref{t:line3} shows that if $GA_1(G) \le 4$ and $GA_1(G) \in \ZZ$, then $G$ is a regular graph.
One can think that perhaps, in general, if $GA_1(G) \in \ZZ$, then $G$ is a regular graph.
The following example shows that this is not true:

\medskip

\noindent {\bf Example 6}.
Given any positive integer $k$,
the complete bipartite graph $K_{5k,20k}$ verifies
$$
GA_1(K_{5k,20k})
= \frac{2(5k \cdot 20k)^{3/2}}{5k+20k}
= \frac{2(10k)^{3}}{25k}
= 80 k^2
\in \ZZ ,
$$
and $K_{5k,20k}$ is not a regular graph.

\section{Geometric-arithmetic index and line graphs}


Recall that if the vertex $a$ in $\mathcal{L}(G)$ corresponds to the edge $uv$, then $d_a=d_u+d_v-2$.
Hence, $\D_{\mathcal{L}(G)}\le 2\D-2$ and $\d_{\mathcal{L}(G)}\ge 2\d-2$.


We need the following technical result.

\begin{lemma} \label{l:line3}
We have for any $\d\ge 1$ and $t \ge 0$ with $2\d+t>2$
$$
\frac{\sqrt{(\d+t-1)(\d-1)}}{2\d+t-2} \le \frac{\sqrt{(\d+t)\d}}{2\d+t} \, .
$$
\end{lemma}

\begin{proof}
The statement is equivalent to
$$
\d(\d+t)(2\d+t-2)^2 - (\d-1)(\d+t-1)(2\d+t)^2
= t^3 +(2\d-1)t^2 \ge 0
$$
for every $\d\ge 1$ and $t \ge 0$, which is direct.
\end{proof}

\begin{proposition} \label{p:line02}
If $G$ is a non-trivial graph with $m$ edges, maximum degree $\D$ and minimum degree $\d$, then
$$
\frac{(M_1(G) - 2m)\sqrt{(\D-1)(\d-1)}}{\D + \d-2}
\le GA_1(\mathcal{L}(G))
\le \frac12\, M_1(G)-m .
$$
The equality is attained in the lower bound if and only if $G$ is regular.
The equality is attained in the upper bound if and only if $\mathcal{L}(G)$ is regular.
\end{proposition}

\begin{proof}
Since $d_{uv}=d_u+d_v-2$ for every $uv \in V(\mathcal{L}(G))=E(G)$, we have
$$
2 m_{\mathcal{L}(G)}
= \sum_{uv\in V(\mathcal{L}(G))} d_{uv}
= \sum_{uv\in E(G)} (d_u + d_v - 2)
= M_1(G) - 2m .
$$
This equality and Proposition \ref{p:0} give
$$
GA_1(\mathcal{L}(G))
\le m_{\mathcal{L}(G)}
= \frac12\, M_1(G)-m .
$$
\indent
The previous argument and Proposition \ref{p:0} give that the equality is attained if and only if $\mathcal{L}(G)$ is regular.

\smallskip

Proposition \ref{p:0} gives
$$
\frac{2m_{\mathcal{L}(G)}\sqrt{\D_{\mathcal{L}(G)}\d_{\mathcal{L}(G)}}}{\D_{\mathcal{L}(G)} + \d_{\mathcal{L}(G)}} \le
GA_1(\mathcal{L}(G)).
$$
Since $2m_{\mathcal{L}(G)} = M_1(G) - 2m$,
$\D_{\mathcal{L}(G)} \le 2\D-2$ and $\d_{\mathcal{L}(G)} \ge 2\d-2$, Corollary \ref{c:t} and Proposition \ref{p:0} give
$$
\frac{(M_1(G) - 2m) \sqrt{(\D-1)(\d-1)}}{\D + \d-2}
\le \frac{2m_{\mathcal{L}(G)}\sqrt{\D_{\mathcal{L}(G)}\d_{\mathcal{L}(G)}}}{\D_{\mathcal{L}(G)} + \d_{\mathcal{L}(G)}}
\le GA_1(\mathcal{L}(G)) .
$$

If $G$ is regular, then $\D-1=\d-1$,
$$
\frac{(M_1(G) - 2m)\sqrt{(\D-1)(\d-1)}}{\D + \d-2}
= \frac12 \, M_1(G) -m
= m_{\mathcal{L}(G)}
= GA_1(\mathcal{L}(G)) ,
$$
and the equality is attained.

The previous argument gives that if the equality is attained, then
$$
GA_1(\mathcal{L}(G))
= \frac{2m_{\mathcal{L}(G)}\sqrt{\D_{\mathcal{L}(G)}\d_{\mathcal{L}(G)}}}{\D_{\mathcal{L}(G)} + \d_{\mathcal{L}(G)}}
= \frac{2m_{\mathcal{L}(G)}\sqrt{(\D-1)(\d-1)}}{\D + \d-2}
\, .
$$
Thus, Corollary \ref{c:t} gives that $\D_{\mathcal{L}(G)} = 2\D-2$ and $\d_{\mathcal{L}(G)} = 2\d-2$.
Also, Proposition \ref{p:0} gives that $\mathcal{L}(G)$ is regular or biregular.

Let us consider $uv \in E(G)=V(\mathcal{L}(G))$ with $d_{uv}=\D_{\mathcal{L}(G)} = 2\D-2$, and so, $d_{u}=d_v = \D$.
Seeking for a contradiction assume that $N(v)=\{u\}$. Hence, $1=d_v = \D$ and the connected component of $G$ containing $uv$ is just an edge, a contradiction.
Therefore, there exists $w \in N(v) \setminus \{u\}$.
Since $d_{uv}=\D_{\mathcal{L}(G)} = 2\D-2$, $vw \in N(uv)$ and $\mathcal{L}(G)$ is regular or biregular,
we have $d_{vw}=\d_{\mathcal{L}(G)} = 2\d-2$ and so $d_v=d_w=\d$.
Thus, $\D=\d$ and $G$ is regular.
\end{proof}

The following result is known.
We include a proof for the sake of completeness.

\begin{proposition} \label{p:line0}
If $G$ is a non-trivial graph with maximum degree $\D$ and minimum degree $\d$,
then $\mathcal{L}(G)$ is regular if and only if each connected component $\{G_1,\dots,G_k\}$ of $G$ is regular or biregular, and
$\D_i+ \d_i = \D+ \d$ for every $1 \le i \le k$, where $\D_i$ and $\d_i$ denote the maximum and minimum degree of $G_i$, respectively, for $1 \le i \le k$.

In particular, if $G$ is a connected non-trivial graph, then $\mathcal{L}(G)$ is regular if and only if $G$ is regular or biregular.
\end{proposition}

\begin{proof}
Assume that $\mathcal{L}(G)$ is regular.
Thus, there exists a constant $k$ with $d_u + d_v - 2 = k$ for every $uv\in E(G)=V(\mathcal{L}(G))$.
Therefore, $d_v = k + 2 - d_u$ for every $u\in V(G)$ and $v\in N(u)$.
If $w\in N(v)$, then $d_w = k + 2 - d_v = k + 2 - ( k + 2 - d_u) = d_u$, and so, the connected component $G_i$ of $G$ containing $u,v,w$ is regular or biregular, and thus,
$\D_i+ \d_i =k+2$.
Fix $1 \le j \le k$ with $\D_j = \D$; then $\d_j = k+2 - \D \le k+2 - \D_i = \d_i$ for every $1 \le i \le k$, and hence, $\d_j = \d$.
Therefore, $k+2 = \D_j+ \d_j = \D + \d$.

If each $G_i$ is regular or biregular, and $\D_i+ \d_i = \D + \d$ for every $1 \le i \le k$, then for every $uv\in E(G)$ we have
$d_{uv} = d_u + d_v - 2 = \D_i+ \d_i-2 = \D + \d -2$.
Thus, $\mathcal{L}(G)$ is regular.
\end{proof}

In 1956, Nordhaus and Gaddum \cite{NG} gave bounds involving the sum of the chromatic number of a graph and its complement.
Motivated by these results, Das obtains in \cite{D} analogous conclusions for the geometric-arithmetic index of a graph and its complement.
Our next theorem is also a Nordhaus-Gaddum-type result for the geometric-arithmetic index of a graph and its line graph.


\begin{corollary} \label{c:line1}
If $G$ is a non-trivial graph with maximum degree $\D$ and minimum degree $\d$, then
$$
\frac{M_1(G)\sqrt{(\D-1)(\d-1)}}{\D + \d-2}
\le GA_1(G) + GA_1(\mathcal{L}(G))
\le \frac12\, M_1(G) ,
$$
and the equality in each inequality is attained if and only if $G$ is regular.
\end{corollary}

\begin{proof}
Propositions \ref{p:0} and \ref{p:line02} give the upper bound.

Proposition \ref{p:line02} gives
$$
GA_1(\mathcal{L}(G))
\ge \frac{(M_1(G) - 2m)\sqrt{(\D-1)(\d-1)}}{\D + \d-2} \,.
$$

Proposition \ref{p:0} and Lemma \ref{l:line3} (with $\d+t=\D$) give
$$
GA_1(G)
\ge \frac{2m\sqrt{\D\d}}{\D + \d}
\ge \frac{2m\sqrt{(\D-1)(\d-1)}}{\D + \d-2} \,,
$$
and we obtain the lower bound.

The previous argument gives that if the equality in the upper bound is attained, then
$$
GA_1(G) = m,
\qquad
GA_1(\mathcal{L}(G)) = m_{\mathcal{L}(G)} .
$$
Thus, Proposition \ref{p:0} gives that $G$ is regular.

Also, if the equality in the lower bound is attained, then
$$
GA_1(G) = \frac{2\sqrt{(\D-1)(\d-1)}}{\D + \d-2} \, m,
\qquad
GA_1(\mathcal{L}(G))
= \frac{(M_1(G) - 2m)\sqrt{(\D-1)(\d-1)}}{\D + \d-2} \,,
$$
and Proposition \ref{p:line02} gives that $G$ is regular.

If $G$ is regular, then $\D-1=\d-1$,
$$
\frac{M_1(G) \sqrt{(\D-1)(\d-1)}}{\D + \d-2}
= \frac12 \, M_1(G)
= m + m_{\mathcal{L}(G)}
= GA_1(G) + GA_1(\mathcal{L}(G)) ,
$$
and the equality is attained.
\end{proof}

The following result provides a lower bound of $GA_1(\mathcal{L}(G))$ involving $GA_1(G)$.
We need a previous result.

\begin{lemma} \label{l:line7}
If a non-trivial connected graph $G$ is not isomorphic to a path graph, then $m \le m_{\mathcal{L}(G)}$.
\end{lemma}

\begin{proof}
If $G$ is a cycle graph, then $\mathcal{L}(G) \approx G$ and $m = m_{\mathcal{L}(G)}$.
If $G$ is not isomorphic to either a path or a cycle graph, then $\D \ge 3$.

Denote by $V_3(G)$ the set of vertices in $u\in V(G)$ with degree $d_u\ge 3$.
If $u\in V_3(G)$, then the edges incident to $u$ correspond to a complete graph $\Gamma_u$ in $\mathcal{L}(G)$ with $d_u$ vertices and
$\frac12\,d_u(d_u-1) \ge d_u$ edges (note that if $u$ and $u'$ are different vertices in $V_3(G)$, then $E(\Gamma_u) \cap E(\Gamma_{u'})= \emptyset$).
Let us define
$E^1(G)=\{uv \in E(G)\,|\; u\in V_3(G)\}$ and $E^1(\mathcal{L}(G)) = \cup_{u\in V_3(G)} \Gamma_u$.
Then
$$
card\,E^1(\mathcal{L}(G))
= \sum_{u\in V_3(G)} card\,E(\G_u)
= \sum_{u\in V_3(G)} \frac12\,d_u(d_u-1)
\ge \sum_{u\in V_3(G)} d_u
\ge card\,E^1(G) .
$$

Let $\p V_3(G)$ be the set of vertices in $G$ at distance $1$ from $V_3(G)$.
Consider now the connected components $G_1,\dots,G_r$ of $G \setminus E^1(G)$.
We have $d_u \le 2$ for every $u \in V(G_j)$ and $1 \le j \le r$.
Since $G$ is not isomorphic to a path graph, then for each $1 \le j \le r$, $G_j$ is a path graph joining either two vertices in $\p V_3(G)$, or a vertex in $\p V_3(G)$ and a vertex with degree $1$.
Denote by $v_j^1, v_j^{2},\dots,v_j^{k_j}$ the vertices in $V(G_j)$ ordered in such a way that
$E(G_j)=\{v_j^1 v_j^{2},v_j^2 v_j^{3},\dots,v_j^{k_j-1} v_j^{k_j}\}$ and $v_jv_j^{k_j} \in E(G)$ for some $v_j \in V_3(G)$.
Furthermore, $d_{v_j^{2}}=\cdots= d_{v_j^{k_j}}=2$.

If $G_j$ is a path graph joining two vertices in $\p V_3(G)$, then $d_{v_j^{1}}=2$ and $v_j'v_j^{1} \in E(G)$ for some $v_j' \in V_3(G)$.
Let $\L_j$ be the set of edges in the path in $\mathcal{L}(G)$ with the $k_j+1$ vertices $\{v_j'v_j^1,v_j^1 v_j^{2},v_j^2 v_j^{3},\dots,v_j^{k_j-1} v_j^{k_j}, v_j^{k_j}v_j\}$.
Thus, $card\,\L_j = k_j=1 + card\,E(G_j) > card\,E(G_j)$.

If $G_j$ is a path graph joining a vertex in $\p V_3(G)$ and a vertex with degree $1$, then $d_{v_j^{1}}=1$.
Let $\L_j$ be the set of edges in the path in $\mathcal{L}(G)$ with the $k_j$ vertices $\{ v_j^1 v_j^{2},v_j^2 v_j^{3},\dots,v_j^{k_j-1} v_j^{k_j}, v_j^{k_j}v_j\}$.
Thus, $card\,\L_j = k_j-1 = card\,E(G_j)$.

Since $\L_i \cap \L_j = \emptyset$ for $1 \le i < j \le r$ and $\L_j \cap E^1(\mathcal{L}(G)) = \emptyset$ for $1 \le j \le r$, we have
$$
m_{\mathcal{L}(G)}
= card\,E^1(\mathcal{L}(G)) + \sum_{j=1}^r card\,\L_j
\ge card\,E^1(G) + \sum_{j=1}^r card\,E(G_j)
= m .
$$
\end{proof}

Lemma \ref{l:line7} has the following consequence.

\begin{corollary} \label{c:line7}
If $G$ is a non-trivial connected graph, then $\mathcal{L}(G)$ is a tree if and only if $G$ is isomorphic to a path graph.
\end{corollary}

\begin{proof}
Since $n_{\mathcal{L}(G)}=m$, if $G$ is not isomorphic to a path graph, then Lemma \ref{l:line7} gives $n_{\mathcal{L}(G)} \le m_{\mathcal{L}(G)}$ and $\mathcal{L}(G)$ is not a tree.
Reciprocally, if $G$ is isomorphic to a path graph $P_{n}$, then $\mathcal{L}(G)$ is isomorphic to the path graph $P_{n-1}$, which is a tree.
\end{proof}

\begin{theorem} \label{t:line6}
If $G$ is a non-trivial graph with maximum degree $\D$ and minimum degree $\d$, then
$$
GA_1(\mathcal{L}(G))
\ge \min \Big\{ \frac{3}{4\sqrt{2}}\,, \;\frac{2\sqrt{(2\D-2)\max \{2\d-2, 1\}}}{2\D-2 + \max \{2\d-2, 1\}} \, \Big\}\, GA_1(G) .
$$
The equality is attained for the path graph $P_3$.
\end{theorem}

\begin{proof}
Assume first that $G$ is a connected graph.

Recall that $\D_{\mathcal{L}(G)} \le 2\D-2$ and $\d_{\mathcal{L}(G)} \ge \max \{2\d-2,\,1\}$.
These inequalities, Proposition \ref{p:0} and Corollary \ref{c:t}
give
$$
GA_1(\mathcal{L}(G))
\ge \frac{2m_{\mathcal{L}(G)}\sqrt{\D_{\mathcal{L}(G)}\d_{\mathcal{L}(G)}}}{\D_{\mathcal{L}(G)} + \d_{\mathcal{L}(G)}}
\ge \frac{2m_{\mathcal{L}(G)}\sqrt{(2\D-2)\max \{2\d-2, 1\}}}{2\D-2 + \max \{2\d-2, 1\}}\, .
$$
Furthermore, Proposition \ref{p:0} gives $GA_1(G)\le m$

If $G$ is not isomorphic to a path graph, then Lemma \ref{l:line7} gives $m \le m_{\mathcal{L}(G)}$ and
$$
GA_1(\mathcal{L}(G))
\ge \frac{2 \sqrt{(2\D-2)\max \{2\d-2, 1\}}}{2\D-2 + \max \{2\d-2, 1\}}\, m
\ge \frac{2 \sqrt{(2\D-2)\max \{2\d-2, 1\}}}{2\D-2 + \max \{2\d-2, 1\}}\,GA_1(G) .
$$

If $G$ is isomorphic to the path graph $P_3$, then $GA_1(G) = \frac{4\sqrt{2}}{3}$ and $GA_1(\mathcal{L}(G)) = GA_1(P_2) = 1$.

If $G$ is isomorphic to the path graph $P_n$ with $n \ge 4$, then $GA_1(G) = n-3 + \frac{4\sqrt{2}}{3}$ and $GA_1(\mathcal{L}(G)) = GA_1(P_{n-1}) = n-4 + \frac{4\sqrt{2}}{3}$.

Hence, if $G$ is isomorphic to a path graph, we have
$$
\frac{GA_1(\mathcal{L}(G))}{GA_1(G)}
\ge \min \Big\{ \frac{3}{4\sqrt{2}}\,, \;\min_{n\ge 4} \frac{n-4 + \frac{4\sqrt{2}}{3}}{n-3 + \frac{4\sqrt{2}}{3}} \, \Big\}
= \min \Big\{ \frac{3}{4\sqrt{2}}\,, \; \frac{ \frac{4\sqrt{2}}{3}}{1 + \frac{4\sqrt{2}}{3}} \, \Big\}
= \frac{3}{4\sqrt{2}} \, .
$$

Therefore, the inequality holds for every non-trivial connected graph, and the equality is attained for the path graph $P_3$, since $\frac{3}{4\sqrt{2}}< \frac{2\sqrt{2}}{3}$.

Finally, assume that $G$ has connected components $G_1,\dots,G_k$.
Denote by $\D_j$ and $\d_j$ the maximum and minimum degree, respectively, of $G_j$ for $1 \le j \le k$.
Since $\D_j\le \D$ and $\d_j \ge \d_j$ for $1 \le j \le k$, Corollary \ref{c:t} gives
$$
\begin{aligned}
GA_1(\mathcal{L}(G))
& = \sum_{j=1}^{k} GA_1(\mathcal{L}(G_j))
\ge \sum_{j=1}^{k} \min \Big\{ \frac{3}{4\sqrt{2}}\,, \;\frac{2\sqrt{(2\D_j-2)\max \{2\d_j-2, 1\}}}{2\D_j-2 + \max \{2\d_j-2, 1\}} \, \Big\}\, GA_1(G_j)
\\
& \ge \sum_{j=1}^{k} \min \Big\{ \frac{3}{4\sqrt{2}}\,, \;\frac{2\sqrt{(2\D-2)\max \{2\d-2, 1\}}}{2\D-2 + \max \{2\d-2, 1\}} \, \Big\}\, GA_1(G_j)
\\
& = \min \Big\{ \frac{3}{4\sqrt{2}}\,, \;\frac{2\sqrt{(2\D-2)\max \{2\d-2, 1\}}}{2\D-2 + \max \{2\d-2, 1\}} \, \Big\}\, GA_1(G) .
\end{aligned}
$$
\end{proof}

We can improve the bound in Theorem \ref{t:line6} for a special class of graphs.

\begin{theorem} \label{t:line7}
Let $G$ be a non-trivial graph such that each connected component of $G$ is regular or biregular and is not isomorphic to $P_3$.
Then $GA_1(\mathcal{L}(G)) \ge GA_1(G)$,
and the equality is attained for every union of cycle graphs.
\end{theorem}

\begin{proof}
By linearity, without loss of generality we can assume that $G$ is connected.
Thus, $G$ is a regular or biregular graph, and $\mathcal{L}(G)$ is a regular graph.

Assume first that $G$ is a regular graph.
Since $G$ is a non-trivial graph, we have $\d\ge 2$.
Since path graphs are not regular, Proposition \ref{p:0} and Lemma \ref{l:line7} give
$GA_1(G) = m \le m_{\mathcal{L}(G)} = GA_1(\mathcal{L}(G))$.

Assume now that $G$ is a $(\D,\d)$-biregular graph.
Denote by $n_1$ and $n_2$ the number of vertices of $G$ with degree $\d$ and $\D$, respectively.
Thus $2m=n_1\d + n_2\D$ and $M_1(G) = n_1\d^2 + n_2\D^2$.

Suppose $\d \ge 2$, then $2(n_1\d + n_2\D) \le n_1\d^2 + n_2\D^2$.
Since $2\sqrt{\D\d} \le \D + \d$,
$$
\begin{aligned}
(n_1\d + n_2\D)(2\sqrt{\D\d} + \D + \d)
& \le (n_1\d + n_2\D)2(\D + \d)
\le (n_1\d^2 + n_2\D^2)(\D + \d) ,
\\
(n_1\d + n_2\D)\frac{2\sqrt{\D\d} + \D + \d}{\D + \d}
& \le n_1\d^2 + n_2\D^2 ,
\\
\frac12\,(n_1\d + n_2\D)\frac{2\sqrt{\D\d}}{\D + \d}
& \le \frac12\,(n_1\d^2 + n_2\D^2) - \frac12\,(n_1\d + n_2\D) ,
\\
GA_1(G) = \frac{2m\sqrt{\D\d}}{\D + \d}
& \le \frac12\, M_1(G) - m
\le m_{\mathcal{L}(G)} = GA_1(\mathcal{L}(G)) .
\end{aligned}
$$

Finally, assume $\d =1$.
Since $G$ is connected, $G$ is isomorphic to the star graph with $n$ vertices $S_n$.
Since $G$ is not isomorphic to $P_3=S_3$, we have $n \ge 4$ and $2 \le n-2$.
Then $\mathcal{L}(G)$ is isomorphic to the complete graph $K_{n-1}$ and
$$
GA_1(G) \le m = n-1
\le \frac12\, (n-1)(n-2)
= m_{\mathcal{L}(G)} = GA_1(\mathcal{L}(G)) .
$$

If $G$ is a cycle graph, then $\mathcal{L}(G) \approx G$ and the equality is attained.
\end{proof}

Note that the inequality $GA_1(G) \le GA_1(\mathcal{L}(G))$ does not hold for path graphs, since
$\mathcal{L}(P_{n}) = P_{n-1}$.

\medskip

In \cite{MH} and \cite{VF} (see also \cite[p.609-610]{DGF}) appear the following inequalities:
\begin{equation} \label{eq23}
GA_1(G) \ge \frac{2(n-1)^{3/2}}{n}\, ,
\qquad
GA_1(G) \ge \frac{2m}{n}\, .
\end{equation}

These inequalities are improved by \cite[Theorem 2.4]{RS2}:

\begin{theorem} \label{t:end}
If $G$ is a graph with $n$ vertices and $m$ edges, then
$$
GA_1(G) \ge \frac{2m\sqrt{n-1}}{n}\,,
$$
and the equality is attained if and only if $G$ is a star graph.
\end{theorem}

The fact $n_{\mathcal{L}(G)} = m$ and \eqref{eq23} have the following consequence.

\begin{corollary} \label{p:line8}
If $G$ is a non-trivial graph with $m$ edges, then
$$
GA_1(\mathcal{L}(G)) \ge \frac{2(m-1)^{3/2}}{m} \,.
$$
\end{corollary}

This result can be improved for almost every graph.
In order to do it, we need the following technical result.

\begin{lemma} \label{l:line8}
For $x_1,\dots,x_k \ge 2$ we have
$$
\sum_{j=1}^k \sqrt{x_j-1} \, \ge \sqrt{\sum_{j=1}^k x_j-1} \,.
$$
\end{lemma}

\begin{proof}
Consider the function
$$
g(x_1,\dots,x_k)
= \sum_{j=1}^k \sqrt{x_j-1}  - \sqrt{\sum_{j=1}^k x_j-1} \,,
$$
with $x_j \ge 2$ for every $1\le j \le k$.
Since
$$
\frac{\p g}{\p x_i}(x_1,\dots,x_k)
= \frac1{2\sqrt{x_i-1}}  - \frac1{2\sqrt{\sum_{j=1}^k x_j-1}} \,.
$$
and $x_j \ge 2$, we have $\p g/\p x_i \ge 0$ for every $1\le i \le k$,
and we conclude
$$
g(x_1,\dots,x_k)
\ge g(2,\dots,2)
= k - \sqrt{2k-1} \ge 0 .
$$
\end{proof}

\begin{theorem} \label{t:line8}
Let $G$ be a non-trivial graph with $m$ edges
such that each connected component  of $G$ is not isomorphic to a path graph $P_n$ with $n\le 6$.
Then
$$
GA_1(\mathcal{L}(G)) \ge 2\sqrt{m-1} \,.
$$
\end{theorem}

\begin{proof}
Assume first that $G$ is a connected graph.
Theorem \ref{t:end} gives
$$
GA_1(\mathcal{L}(G))
\ge \frac{2m_{\mathcal{L}(G)} \sqrt{n_{\mathcal{L}(G)}-1}}{n_{\mathcal{L}(G)}}
= \frac{2m_{\mathcal{L}(G)} \sqrt{m-1}}{m}\,.
$$
By Lemma \ref{l:line7},
if $G$ is not isomorphic to a path graph, then
$m_{\mathcal{L}(G)} \ge m$, and the conclusion holds.

Assume now that $G$ is isomorphic to a path graph $P_n$ with $n \ge 7$.
Therefore, $\mathcal{L}(G)=P_{n-1}=P_m$ with $m \ge 6$.
Then the inequality is equivalent to
$m-3+\frac{4\sqrt{2}}{3} \ge 2\sqrt{m-1} \,,$
and one can easily check that it holds for every $m \ge 6$.

Finally, assume that $G$ has connected components $G_1,\dots,G_k$ with $m_1,\dots,m_k$ edges, respectively.
We have proved that
$$
GA_1(\mathcal{L}(G))
= \sum_{j=1}^k GA_1(\mathcal{L}(G_j))
\ge 2 \sum_{j=1}^k \sqrt{m_j-1} \,.
$$
Since $G_j$ is a non-trivial connected graph which is not isomorphic to $P_3$, we have $m_j \ge 2$ for every $1 \le j \le k$.
Thus, Lemma \ref{l:line8} gives
$$
GA_1(\mathcal{L}(G))
\ge 2 \sum_{j=1}^k \sqrt{m_j-1}
\ge 2 \sqrt{m -1} \,.
$$
\end{proof}

Note that Theorems \ref{t:end} and \ref{t:line8} and Corollary \ref{p:line8} provide another Nordhaus-Gaddum-type result for the geometric-arithmetic index of a graph and its line graph.

\smallskip

In the paper \cite{GT}, where Zagreb indices were introduced, the \emph{forgotten topological index} (or \emph{F-index}) is defined as
$$
F(G) = \sum_{u \in V(G)} d_u^3.
$$
Both the forgotten
topological index and the first Zagreb index were employed in the formulas for total $\pi$-electron energy in \cite{GT},
as a measure of branching extent of the carbon-atom skeleton of the underlying molecule.
However, this index never got attention
except recently, when Furtula and Gutman in \cite{3} established some basic properties
of the F-index and showed that its predictive ability is almost similar to
that of first Zagreb index and for the entropy and acetic factor, both of them yield correlation
coefficients greater than $0.95$.
Besides, \cite{3} pointed out the importance of the F-index: it can be used to obtain a high accuracy of the prediction of logarithm of the
octanol-water partition coefficient (see also \cite{5}).
The extremal trees with respect to the F-index have been investigated in \cite{5}.
Furthermore, several papers contain more lower and upper bounds for the forgotten index (see, e.g., \cite{ChCh}, \cite{RS4}).

\begin{theorem} \label{t:line10}
If $G$ is a non-trivial graph with $m$ edges, then
$$
M_1(\mathcal{L}(G))
= 4m - 4M_1(G) +2 M_2(G) + F(G) .
$$
\end{theorem}

\begin{proof}
Since the vertex in $\mathcal{L}(G)$ corresponding to $uv\in E(G)$ has degree $d_u + d_v - 2$, we have
$$
\begin{aligned}
M_1(\mathcal{L}(G))
& = \sum_{uv\in E(G)}(d_u + d_v - 2)^2
\\
& = \sum_{uv\in E(G)}(d_u^2 + d_v^2) +2 \!\!\!\! \sum_{uv\in E(G)} d_u d_v - 4 \!\!\!\! \sum_{uv\in E(G)}(d_u + d_v) + \!\!\! \sum_{uv\in E(G)} 4
\\
& = \sum_{u\in V(G)} d_u^3 +2 M_2(G) - 4M_1(G) + 4m
\\
& = F(G) +2 M_2(G) - 4M_1(G) + 4m .
\end{aligned}
$$
\end{proof}

Since $F(G) \ge \d M_1(G)$,
Theorem \ref{t:line10} has the following consequence.

\begin{corollary} \label{c:line10}
If $G$ is a non-trivial graph with $m$ edges and minimum degree $\d$, then
$$
M_1(\mathcal{L}(G))
\ge 4m + (\d- 4)M_1(G) +2 M_2(G)  .
$$
\end{corollary}

The next result appears in \cite{gaCadiz}.

\begin{lemma} \label{t:pi1tris}
If $\a>0$ and $G$ is a graph with $m$ edges, maximum degree $\D$ and minimum degree $\d$, then
$$
GA_1(G)
\ge \frac{2^{1/(2\a)} \d^{1/2}m^{(2\a+1)/(2\a)}}{\D \, M_1^{1-\a}(G)^{1/(2\a)}}\,,
$$
and the equality holds for some $\a$ if and only if $G$ is regular.
\end{lemma}

\begin{theorem} \label{c:pi1tris}
If $\a>0$ and $G$ is a non-trivial graph with $m$ edges, maximum degree $\D$ and minimum degree $\d$, then
$$
\begin{aligned}
GA_1(\mathcal{L}(G))
\ge \frac{ (2\d-2)^{1/2} \D^{(1-\a)/(2\a)} \big( M_1(G)-2m \big)^{(2\a+1)/(2\a)}}{4 \, (\D-1)^{(1+\a)/(2\a)}\chi_{_{1-\a}}(G)^{1/(2\a)}} \,,
& \qquad \text{ if } \, 0< \a \le 1,
\\
GA_1(\mathcal{L}(G))
\ge \frac{ \sqrt{2} \, (\d-1)^{1/(2\a)} \d^{(\a-1)/(2\a)} \big( M_1(G)-2m \big)^{(2\a+1)/(2\a)}}{4 \, (\D-1) \chi_{_{1-\a}}(G)^{1/(2\a)}} \,,
& \qquad \text{ if } \, \a > 1.
\end{aligned}
$$
\end{theorem}

\begin{proof}
For any $\a>0$, Lemma \ref{t:pi1tris} gives
$$
GA_1(\mathcal{L}(G))
\ge \frac{2^{1/(2\a)} \d_{\mathcal{L}(G)}^{1/2}m_{\mathcal{L}(G)}^{(2\a+1)/(2\a)}}{\D_{\mathcal{L}(G)} \, M_1^{1-\a}(\mathcal{L}(G))^{1/(2\a)}}\,.
$$
Since
$\D_{\mathcal{L}(G)} \le 2\D-2$ and $\d_{\mathcal{L}(G)} \ge 2\d-2$, we have
$$
GA_1(\mathcal{L}(G))
\ge \frac{2^{1/(2\a)} (2\d-2)^{1/2}\big( \frac12\, M_1(G)-m \big)^{(2\a+1)/(2\a)}}{(2\D-2) M_1^{1-\a}(\mathcal{L}(G))^{1/(2\a)}}\,.
$$

If $0< \a \le 1$, then
$$
\begin{aligned}
M_1^{1-\a}(\mathcal{L}(G))
& = \sum_{uv \in V(\mathcal{L}(G))} d_{uv}^{1-\a}
= \sum_{uv \in E(G)} (d_u + d_{v}-2)^{1-\a}
\\
& \le \Big( \frac{\D-1}{\D} \Big)^{1-\a} \!\!\! \sum_{uv \in E(G)} (d_u + d_{v})^{1-\a}
= \Big( \frac{\D-1}{\D} \Big)^{1-\a} \chi_{_{1-\a}}(G),
\end{aligned}
$$
and we conclude
$$
\begin{aligned}
GA_1(\mathcal{L}(G))
& \ge \frac{ (2\d-2)^{1/2}\big( M_1(G)-2m \big)^{(2\a+1)/(2\a)}}{(4\D-4) \, \big( \big( \frac{\D-1}{\D} \big)^{1-\a}\chi_{_{1-\a}}(G)\big)^{1/(2\a)}}
\\
& = \frac{ (2\d-2)^{1/2} \D^{(1-\a)/(2\a)} \big( M_1(G)-2m \big)^{(2\a+1)/(2\a)}}{4 \, (\D-1)^{(1+\a)/(2\a)}\chi_{_{1-\a}}(G)^{1/(2\a)}}
\,.
\end{aligned}
$$

If $\a > 1$, then
$$
\begin{aligned}
M_1^{1-\a}(\mathcal{L}(G))
& = \sum_{uv \in V(\mathcal{L}(G))} d_{uv}^{1-\a}
= \sum_{uv \in E(G)} (d_u + d_{v}-2)^{1-\a}
\\
& \le \Big( \frac{\d}{\d-1} \Big)^{1-\a} \!\!\! \sum_{uv \in E(G)} (d_u + d_{v})^{1-\a}
= \Big( \frac{\d}{\d-1} \Big)^{1-\a} \chi_{_{1-\a}}(G),
\end{aligned}
$$
and we have
$$
\begin{aligned}
GA_1(\mathcal{L}(G))
& \ge \frac{ \sqrt{2} \, (\d-1)^{1/2}\big( M_1(G)-2m \big)^{(2\a+1)/(2\a)}}{4 \,(\D-1) \, \big( \big( \frac{\d}{\d-1} \big)^{1-\a}\chi_{_{1-\a}}(G)\big)^{1/(2\a)}}
\\
& = \frac{ \sqrt{2} \, (\d-1)^{1/(2\a)} \d^{(\a-1)/(2\a)} \big( M_1(G)-2m \big)^{(2\a+1)/(2\a)}}{4 \, (\D-1) \chi_{_{1-\a}}(G)^{1/(2\a)}}
\,.
\end{aligned}
$$
\end{proof}

\begin{theorem} \label{t:gam20}
If $G$ is a graph with $m$ edges and minimum degree $\d$, then
$$
GA_1(G)
\ge \frac{2 \, \d^{1/2} m^{2}}{M_1^{3/2}(G)} \,,
$$
and the equality is attained if and only if $G$ is regular.
\end{theorem}

\begin{proof}
Cauchy-Schwarz inequality gives
$$
\begin{aligned}
m
& = \sum_{uv\in E(G)} \Big(\frac{\sqrt{d_u d_v}}{d_u+d_v} \Big)^{1/2} \Big(\frac{d_u+d_v}{\sqrt{d_u d_v}} \Big)^{1/2}
\\
& \le \Big(\sum_{uv\in E(G)}\frac{\sqrt{d_u d_v}}{d_u+d_v} \Big)^{1/2} \Big(\sum_{uv\in E(G)}\Big(\sqrt{\frac{d_u}{d_v}} + \sqrt{\frac{d_v}{d_u}} \; \Big) \Big)^{1/2},
\end{aligned}
$$
Since
$$
\begin{aligned}
\sum_{uv\in E(G)}\Big(\sqrt{\frac{d_u}{d_v}} + \sqrt{\frac{d_v}{d_u}} \; \Big)
& \le \sum_{uv\in E(G)} \frac{d_u^{1/2} + d_v^{1/2}}{\d^{1/2}}
= \d^{-1/2} \sum_{uv\in E(G)} \big( d_u^{1/2} + d_v^{1/2} \big)
\\
& = \d^{-1/2} \sum_{u\in V(G)} d_u^{1/2}d_u
= \d^{-1/2} M_1^{3/2}(G),
\end{aligned}
$$
we have
$$
\begin{aligned}
m^2
& \le \frac1{2} \, GA_1(G) \Big( \d^{-1/2} M_1^{3/2}(G) \Big) ,
\\
GA_1(G)
& \ge \frac{2\,\d^{1/2}m^{2}}{M_1^{3/2}(G)} \,.
\end{aligned}
$$


If the graph is regular, then
$$
\frac{2\d^{1/2}m^{2}}{M_1^{3/2}(G)}
= \frac{2\d^{1/2}m^{2}}{\d^{3/2} n}
= \frac{2m^{2}}{\d n}
= m
= GA_1(G).
$$
\indent
If the equality is attained, then
the previous argument gives $d_v =\d$ for every $v \in V(G)$, and $G$ is a regular graph.
\end{proof}


\begin{theorem} \label{c:gam20bislinea}
If $G$ is a non-trivial graph with $m$ edges, maximum degree $\D$ and minimum degree $\d$, then
$$
GA_1(\mathcal{L}(G))
\ge \frac{(2\d-2)^{1/2} {\D}^{3/2} \big( M_1(G)-2m \big)^{2}}{2 (\D-1)^{3/2}\chi_{_{3/2}}(G)} \,.
$$
\end{theorem}

\begin{proof}
Theorem \ref{t:gam20} and $\d_{\mathcal{L}(G)} \ge 2\d-2$ give
$$
GA_1(\mathcal{L}(G))
\ge \frac{2 \, \d_{\mathcal{L}(G)}^{1/2} m_{\mathcal{L}(G)}^{2}}{M_1^{3/2}(\mathcal{L}(G))}
\ge \frac{2 \, (2\d-2)^{1/2} \big( \frac12\, M_1(G)-m \big)^{2}}{M_1^{3/2}(\mathcal{L}(G))} \,.
$$
Since
$$
\begin{aligned}
M_1^{3/2}(\mathcal{L}(G))
& = \sum_{uv \in V(\mathcal{L}(G))} d_{uv}^{3/2}
= \sum_{uv \in E(G)} (d_u + d_{v}-2)^{3/2}
\\
& \le \Big( \frac{\D-1}{\D} \Big)^{3/2} \!\!\! \sum_{uv \in E(G)} (d_u + d_{v})^{3/2}
= \Big( \frac{\D-1}{\D} \Big)^{3/2} \chi_{_{3/2}}(G),
\end{aligned}
$$
we conclude
$$
GA_1(\mathcal{L}(G))
\ge \frac{(2\d-2)^{1/2} \big( M_1(G)-2m \big)^{2}}{2 \big( \frac{\D-1}{\D} \big)^{3/2}\chi_{_{3/2}}(G)} \,.
$$
\end{proof}


\begin{thebibliography}{99}

\bibitem{5} H. Abdo, D. Dimitrov, I. Gutman, On extremal trees with respect to the F-index,
{\it Kuwait J. Sci.} {\bf 44:3} (2017) 1--8.

\bibitem{AP} V. Andova, M. Petrusevski, Variable Zagreb Indices and Karamata’s Inequality,
{\it MATCH Commun. Math. Comput. Chem.\/} {\bf 65} (2011) 685--690.



\bibitem{ChCh} Z. Che, Z. Chen, Lower and Upper Bounds of the Forgotten Topological Index,
{\it MATCH Commun. Math. Comput. Chem.\/} {\bf 76} (2016) 635--648.


\bibitem{D} K. C. Das, On geometric-arithmetic index of graphs, \emph{MATCH Commun. Math. Comput. Chem.} {\bf 64} (2010) 619--630.

\bibitem{DGF} K. C. Das, I. Gutman, B. Furtula, Survey on Geometric-Arithmetic Indices of Graphs,
\emph{MATCH Commun. Math. Comput. Chem.} {\bf 65} (2011) 595--644.

\bibitem{DGF2} K. C. Das, I. Gutman, B. Furtula, On first geometric-arithmetic index of graphs,
\emph{Discrete Appl. Math.} {\bf 159} (2011) 2030--2037.


\bibitem{DEB} H. Deng, S. Elumalai, S. Balachandran,
Maximum and Second Maximum of Geometric–Arithmetic Index of Tricyclic Graphs, {\it MATCH Commun. Math. Comput. Chem.} {\bf 79} (2018) 467--475.


\bibitem{DZT} Z. Du, B. Zhou, N. Trinajsti\'c, On geometric–arithmetic indices of (molecular) trees, {\it MATCH Commun. Math. Comput. Chem.} {\bf 66} (2011) 681--697.

\bibitem{EIG} M. Eliasi, A. Iranmanesh, I. Gutman, Multiplicative versions of first Zagreb index, {\it MATCH Commun. Math. Comput. Chem.} {\bf 68(1)} (2012) 217--230.

\bibitem{HHD} N. H. M. Husin, R. Hasni, Z. Du, On extremum geometric–arithmetic indices of (molecular) trees, {\it MATCH Commun. Math. Comput. Chem.} {\bf 78} (2017) 375--386.




\bibitem{3} B. Furtula, I. Gutman, A forgotten topological index, {\it J. Math. Chem.\/} {\bf 53 (4)} (2015) 1184--1190.

%

\bibitem{Gutman} I. Gutman, Degree--based topological indices, {\it Croat. Chem. Acta} {\bf 86} (2013) 351--361.

\bibitem{GD} I. Gutman, K. C. Das, The first Zagreb index 30 years after, {\it MATCH Commun. Math. Comput. Chem.} {\bf 50} (2004) 83--92.



\bibitem{Gutman8}
I. Gutman, J. To\v{s}ovi\'c, Testing the quality of molecular structure descriptors. Vertex--degreebased topological indices, \emph{J. Serb. Chem. Soc.} {\bf 78(6)} (2013) 805--810.

\bibitem{GT} I. Gutman, N. Trinajsti\'c, Graph theory and molecular orbitals. Total $\pi$-electron
energy of alternant hydrocarbons, \emph{Chem. Phys. Lett.} {\bf 17} (1972) 535--538.

\bibitem{HN} F. Harary, R. Z. Norman, Some properties of line digraphs,
{\it Rend. Circ. Math. Palermo} {\bf 9} (1960) 161--169.


\bibitem{H} B. Hollas, On the variance of topological indices that depend on the degree of a vertex,
\emph{MATCH Commun. Math. Comput. Chem.} {\bf 54} (2005) 341--350.




\bibitem{Kr} J. Krausz, D\'emonstration nouvelle d'un th\'eor\`eme de Whitney sur les r\'eseaux,
{\it Mat. Fiz. Lapok} {\bf 50} (1943) 75--85.

%
%

\bibitem{LZheng} X. Li, J. Zheng, A unified approach to the extremal trees for different indices,
{\it MATCH Commun. Math. Comput. Chem.\/} {\bf 54} (2005) 195--208.

\bibitem{LZhao} X. Li, H. Zhao, Trees with the first smallest and largest generalized topological indices,
{\it MATCH Commun. Math. Comput. Chem.\/} {\bf 50} (2004) 57--62.


\bibitem{LL} M. Liu, B. Liu, Some properties of the first general Zagreb index,
{\it Australas. J. Combin.\/} {\bf 47} (2010) 285--294.

\bibitem{Liu13}
J. Liu, Q. Zhang, Remarks on harmonic index of graphs, \emph{Utilitas Math.} {\bf 88} (2012) 281--285.

%



\bibitem{MN} A. Mili\v{c}evi\'c, S. Nikoli\'c, On variable Zagreb indices, \emph{Croat. Chem. Acta} {\bf 77} (2004) 97--101.

\bibitem{MH} M. Mogharrab, G. H. Fath-Tabar, Some bounds on $GA_1$ index of graphs,
\emph{MATCH Commun. Math. Comput. Chem.} {\bf 65} (2010) 33--38.

%

\bibitem{NMTJ} S. Nikoli\'c, A. Mili\v{c}evi\'c, N. Trinajsti\'c, A. Juri\'c,
On Use of the Variable Zagreb $^\nu M_2$ Index in QSPR: Boiling Points of Benzenoid Hydrocarbons
{\it Molecules\/} {\bf 9} (2004) 1208--1221.

\bibitem{NKMT} S. Nikoli\'c, G. Kova\v{c}evi\'c, A. Mili\v{c}evi\'c, N. Trinajsti\'c,
The Zagreb Indices 30 years after, \emph{Croat. Chem. Acta} {\bf 76} (2003) 113--124.

\bibitem{NG} E. A. Nordhaus, J. W. Gaddum, On complementary graphs, \emph{Am. Math. Monthly} {\bf 63} (1956) 175--177.



\bibitem{R2} M. Randi\'c, Novel graph theoretical approach to heteroatoms in QSAR,
{\it Chemometrics Intel. Lab. Syst.\/} {\bf 10} (1991) 213--227.


\bibitem{RPL} M. Randi\'c, D. Plav\v{s}i\'c, N. Ler\v{s}, Variable connectivity index for cycle-containing structures,
{\it J. Chem. Inf. Comput. Sci.\/} {\bf 41} (2001) 657--662.




\bibitem{RSS} J. M. Rodr{\'\i}guez, J. L. S\'anchez, J. M. Sigarreta,
On the first general Zagreb index, {\it J. Math. Chem.}, in press,
DOI: 10.1007$/$s10910-017-0816-y

\bibitem{RS2} J. M. Rodr{\'\i}guez, J. M. Sigarreta, On the Geometric-Arithmetic Index,
\emph{MATCH Commun. Math. Comput. Chem.} {\bf 74} (2015) 103--120.

\bibitem{RS3} J. M. Rodr{\'\i}guez, J. M. Sigarreta, Spectral properties of geometric-arithmetic index,
\emph{Appl. Math. Comput.} {\bf 277} (2016) 142--153.

\bibitem{RS4} J. M. Rodr{\'\i}guez, J. M. Sigarreta, New Results on the Harmonic Index and Its Generalizations,
{\it MATCH Commun. Math. Comput. Chem.\/} {\bf 78(2)} (2017), 387--404.

\bibitem{gaCadiz} J. M. Rodr{\'\i}guez, J. M. Sigarreta, Relations between some topological indices. Submitted.


\bibitem{RLC} P. S. Ranjini, V. Lokesha, I. N. Cang\"ul, On the Zagreb indices of the line graphs of the subdivision graphs, {\it Appl. Math. Comput.\/} {\bf 218} (2011) 699--702.

\bibitem{SX} G. Su, L. Xu, Topological indices of the line graph of subdivision graphs and their Schur bounds, {\it Appl. Math. Comput.\/} {\bf 253} (2015) 395--401.


\bibitem{TRC} TRC Thermodynamic Tables. Hydrocarbons; Thermodynamic Research Center,
The Texas A $\&$ M University System: College Station, TX, 1987.



\bibitem{SDGM} M. Singh, K. Ch. Das, S. Gupta, A. K. Madan,
Refined variable Zagreb indices: highly discriminating topological descriptors for QSAR/QSPR,
{\it Int. J. Chem. Modeling\/} {\bf 6(2-3)} 403--428.

\bibitem{VGJ} M. V\"oge, A. J. Guttmann, I. Jensen, On the number of benzenoid hydrocarbons, \emph{J. Chem. Inf. Comput. Sci.} {\bf 42} (2002) 456--466.

\bibitem{VF} D. Vuki\v{c}evi\'c, B. Furtula, Topological index based on the ratios of geometrical and
arithmetical means of end-vertex degrees of edges, \emph{J. Math. Chem.} {\bf 46} (2009) 1369--1376.


\bibitem{W} H. Whitney, Congruent graphs and the connectivity of graphs, {\it Amer. J. Math.} {\bf 54} (1932) 150--168.










\bibitem{ZT1} B. Zhou, N. Trinajsti\'c, On a novel connectivity index, {\it J. Math. Chem.\/} {\bf 46} (2009) 1252--1270.

\bibitem{ZT2} B. Zhou, N. Trinajsti\'c, On general sum-connectivity index, \emph{J. Math. Chem.} {\bf 47} (2010) 210--218.

\end{thebibliography}
\end{document}